\documentclass[12pt,reqno]{amsart}

\usepackage{fancyvrb}

\usepackage{amssymb, amsmath, amsthm}
\usepackage[bookmarks, colorlinks=true, linkcolor=blue, citecolor=blue, urlcolor=blue]{hyperref}
\usepackage[alphabetic,lite,nobysame]{amsrefs}
\usepackage{verbatim}
\usepackage{amscd}   
\usepackage[all]{xy} 
\usepackage{youngtab} 
\usepackage{young} 
\usepackage{ytableau} 
\usepackage{tikz}
\usepackage{ mathrsfs }
\usepackage{cases}
\usepackage{verbatimbox}
\usepackage{tabu}
\usepackage{multicol}

\usepackage[margin=1in]{geometry}

\newcommand{\defi}[1]{{\upshape\bf \sffamily #1}}

\renewcommand{\a}{\alpha}
\renewcommand{\b}{\beta}

\newcommand{\D}{\operatorname{D}}
\newcommand{\bw}{\bigwedge}
\renewcommand{\det}{\textrm{det}}

\newcommand{\onto}{\twoheadrightarrow}
\newcommand{\oo}{\otimes}

\newcommand{\chr}{\operatorname{char}}
\newcommand{\coker}{\operatorname{Coker}}
\newcommand{\Ext}{\operatorname{Ext}}

\newcommand{\Hom}{\operatorname{Hom}}

\newcommand{\SL}{\operatorname{SL}}

\newcommand{\Sym}{\operatorname{Sym}}
\newcommand{\Tor}{\operatorname{Tor}}

\newcommand{\rH}{\mathrm{H}}

\newcommand{\bb}[1]{\mathbb{#1}}

\newcommand{\op}[1]{\operatorname{#1}}
\renewcommand{\rm}[1]{\textrm{#1}}
\newcommand{\mc}[1]{\mathcal{#1}}
\newcommand{\mf}[1]{\mathfrak{#1}}
\newcommand{\ms}[1]{\mathscr{#1}}

\def\kk{{\mathbf k}}
\def\PP{{\textbf P}}

\def\lra{\longrightarrow}

\newtheorem{theorem}[equation]{Theorem}
\newtheorem*{theorem*}{Theorem}
\newtheorem{lemma}[equation]{Lemma}

\newtheorem{proposition}[equation]{Proposition}
\newtheorem{corollary}[equation]{Corollary}
\newtheorem*{corollary*}{Corollary}

\newtheorem*{main*}{Main Theorem}

\theoremstyle{definition}
\newtheorem{defn}[equation]{Definition}

\newtheorem{eg}[equation]{Example}

\theoremstyle{remark}
\newtheorem{rmk}[equation]{Remark}
\newenvironment{remark}[1][]{\begin{rmk}[#1] \pushQED{\qed}}{\popQED \end{rmk}}

\numberwithin{equation}{section}

\begin{document}

\title[Bi-graded Koszul modules, K3 carpets, and Green's conjecture]{Bi-graded Koszul modules, K3 carpets, \\and Green's conjecture}

\author{Claudiu Raicu}
\address{Department of Mathematics, University of Notre Dame, 255 Hurley, Notre Dame, IN 46556\newline
\indent Institute of Mathematics ``Simion Stoilow'' of the Romanian Academy}
\email{craicu@nd.edu}
\thanks{CR was supported by NSF DMS-1901886 and a Sloan research fellowship.}

\author{Steven V Sam}
\address{Department of Mathematics, University of California, San Diego, La Jolla, CA 92093}
\email{ssam@ucsd.edu}
\thanks{SS was supported by NSF DMS-1849173 and a Sloan research fellowship.}


\date{August 2, 2021}


\begin{abstract}
We extend the theory of Koszul modules to the bi-graded case, and prove a vanishing theorem that allows us to show that the Canonical Ribbon Conjecture of Bayer and Eisenbud holds over a field of characteristic zero or at least equal to the Clifford index. Our results confirm a conjecture of Eisenbud and Schreyer regarding the characteristics where the generic statement of Green's conjecture holds. They also recover and extend to positive characteristics the results of Voisin asserting that Green's Conjecture holds for generic curves of each gonality.
\end{abstract}

\maketitle

\section{Introduction}\label{sec:intro}

One of the most influential open problems in the study of syzygies over the past $35$ years, which remains open to this date, is Green's Conjecture on canonical curves \cite[Conjecture~5.1]{green}. It asserts that for a smooth curve $C$ of genus $g$ in characteristic zero, the (non-)vanishing behavior of the Koszul cohomology groups $K_{p,1}(C,\omega_C)$, where $\omega_C$ is the canonical bundle, detects the Clifford index of~$C$:
\[ K_{i,1}(C,\omega_C)\neq 0 \Longleftrightarrow i\leq g-1-\op{Cliff}(C).\]
The implication ``$\Longleftarrow$" was proved by Green and Lazarsfeld in \cite[Appendix]{green}, and the converse amounts by duality to showing that
\begin{equation}\label{eq:green-vanishing}
 K_{i,2}(C,\omega_C)=0\mbox{ for }i<\op{Cliff}(C).
\end{equation}
It was soon realized that due to the semi-continuity property of syzygies, one can try to prove generic versions of Green's Conjecture by constructing examples of curves that exhibit the vanishing~(\ref{eq:green-vanishing}). Moreover, singular examples of such curves are good enough as long as they are smoothable. Despite some appealing candidates being proposed over the years (such as rational cuspidal curves, nodal curves, ribbons), the vanishing (\ref{eq:green-vanishing}) for generic curves (where $\op{Cliff}(C) = \lfloor(g-1)/2\rfloor$) remained open until the tour de force by Voisin \cites{voisin-even,voisin-odd} that used cohomology calculations on Hilbert schemes and the geometry of K3 surfaces. The work of Voisin shows~(\ref{eq:green-vanishing}) for a generic curve of any gonality $d$ (where $\op{Cliff}(C) = d-2$), extending earlier results that were established in large genus: $g>(d-1)\cdot(d-2)$ in \cite{schreyer-large-g}, or $g\geq 3d+2$ in \cite{teixidor}.  Building on \cites{HR,voisin-odd}, Aprodu describes, inside each $d$-gonal stratum, explicit loci where Green's conjecture holds \cite{aprodu}. 

More recently, a more elementary and algebraic approach using the theory of Koszul modules has been used in \cite{AFPRW} to prove (\ref{eq:green-vanishing}) for rational cuspidal curves, fulfilling one of the early proposals \cite[Section~3.I]{eisenbud-orientation} and recovering Green's conjecture for generic curves. The main goal of our paper is to extend the theory of Koszul modules to the bi-graded setting, and verify (\ref{eq:green-vanishing}) for rational ribbons, proving the Canonical Ribbon Conjecture \cite{bayer-eisenbud} and recovering Green's conjecture for generic curves in each gonality. We note that a proof of the Canonical Ribbon Conjecture that builds on the work of Voisin was obtained recently by Deopurkar \cite{deopurkar}. 

An important advantage of the approach through Koszul modules is that the methods carry over to positive characteristic. As stated, Green's Conjecture was known to fail in small characteristics even for generic curves, by work of Schreyer \cite{schreyer}, for instance in genus~$7$ and characteristic~$2$. It is then natural to try to identify the appropriate characteristic assumptions to insure that Green's Conjecture remains valid (we note that Bopp and Schreyer have proposed a modification of the conjecture that is characteristic free \cite{bopp-schreyer}, but we won't pursue it here). Eisenbud and Schreyer investigated further this problem in \cite{ES-K3carpets} and conjectured that (\ref{eq:green-vanishing}) should hold for generic curves in characteristic $\geq \lfloor (g-1)/2 \rfloor$. Our results confirm this conjecture, and improve on the lower bound $\geq (g+2)/2$ from \cite{AFPRW}. We also note that the restrictions on the characteristic have a clean explanation in our approach, coming from the fact that symmetric and divided powers are not isomorphic as functors in small characteristics.

\medskip

We now formulate our results more precisely. Throughout this article we work over an algebraically closed field $\kk$. We fix positive integers $a,b$, and let $\mc{S}(a,b)\subset\bb{P}^{a+b+1}$ denote the \defi{rational normal scroll} of type $(a,b)$. By \cite[Theorem~1.3]{gal-pur} there is a unique double structure on $\mc{S}(a,b)$ that is numerically a K3 surface; it is denoted $\mc{X}(a,b)$ and called a \defi{K3 carpet}. Our interest in the study of K3 carpets lies in the fact that their hyperplane sections are canonical ribbons of genus $g=a+b+1$ and Clifford index $\min(a,b)$, and as such they are degenerations of smooth canonical curves with the same invariants (see \cites{bayer-eisenbud,fong,eisenbud-green,ES-K3carpets} and Section~\ref{sec:green}). We will prove:

\begin{theorem}\label{thm:syz-Xab}
 Let $R=\kk[\bb{P}^{g}]$ and $A=\kk[\mc{X}(a,b)]$ denote the homogeneous coordinate rings of $\bb{P}^g$ and $\mc{X}(a,b)$ respectively. If $p=\chr(\kk)$ satisfies $p=0$ or $p\geq\min(a,b)$ then
 \[ \Tor_i^R(A,\kk)_{i+2} = 0 \mbox{ for all }i<\min(a,b).\]
\end{theorem}

By passing to a hyperplane section (and assuming $a\leq b$ in the theorem above) we obtain:

\begin{theorem}[The Canonical Ribbon Conjecture]\label{thm:CRC}
 Let $C$ be a rational ribbon of genus $g$ and Clifford index $a$. If $p=\chr(\kk)$ satisfies $p=0$ or $p\geq a$ then
 \[K_{i,2}(C,\omega_C) = 0\mbox{ for all }i<a.\]
\end{theorem}

\begin{corollary}
  Over a field $\kk$ of characteristic $p=0$ or $p \ge a$, Green's conjecture is true for a non-empty Zariski open subset  inside the locus of genus $g$ curves with Clifford index $a$.
\end{corollary}

Specializing to the case when the Clifford index is generic, $a= \lfloor (g-1)/2 \rfloor$, we confirm the following conjecture of Eisenbud and Schreyer \cite[Conjecture 0.1]{ES-K3carpets}.

\begin{theorem}\label{thm:generic-Green}
 Green's Conjecture is true for a general curve of genus $g$ over a field $\kk$ of characteristic $p=0$ or $p\geq \lfloor (g-1)/2 \rfloor$.
\end{theorem}

The main new idea in our paper is the introduction and use of bi-graded Koszul modules, which we explain next. In the singly-graded case, Koszul modules were introduced by Papadima and Suciu in \cite{PS}, and they have been used in \cite{AFPRW} to prove the Generic Green's Conjecture via degeneration to cuspidal curves. It is this latter paper that constitutes the inspiration for our work. The motivation of Papadima and Suciu for defining Koszul modules comes from geometric group theory, where various incarnations of these modules have been used to great effect by Sullivan, Dimca, Papadima, Suciu, Hain and many others. In this setting, a more familiar name for Koszul modules is that of \defi{infinitesimal Alexander invariants} \cite[Section~1.8]{PS-chen}. For new applications in this context and a more extensive survey of the relevant literature, the reader can consult \cite{AFPRW2}. Although we do not pursue this line of thought here, it is reasonable to expect, and worthwhile to pursue, analogous applications in geometric group theory for the algebraic results on bi-graded Koszul modules that we develop here.

To recall the definition of the singly-graded Koszul modules, we consider a vector space $V$ and form the polynomial ring $S=\Sym(V)$, endowed with the natural grading where the elements of~$V$ have degree one. For a subspace $K\subset \bw^2 V$, we 
form the $3$-term complex
\begin{equation}\label{eq:kosz-sequence}
 K \oo S \lra V\oo S \lra S,
\end{equation}
obtained by replacing $\bw^2 V\oo S$ in the Koszul complex with the submodule $K\oo S$. The \defi{Koszul module $W(V,K)$} is the middle homology of \eqref{eq:kosz-sequence}. It was shown in \cite{PS} that $W(V,K)$ is a finite length module if and only if the orthogonal complement $K^\perp \subset (\bw^2 V)^\vee$ does not contain any non-zero decomposable tensors $a\wedge b$, with $a,b\in V^{\vee}$ (equivalently, the projectivization of $K^{\perp}$ does not intersect the Grassmannian $\rm{Gr}(2,V^{\vee})$). Thinking of $(\bw^2 V)^\vee$ as the subspace of skew-symmetric tensors in $V^{\vee}\oo V^{\vee}$, the decomposable elements $a\wedge b$ precisely correspond to rank two tensors.

For the  bi-graded setting, we assume that $V$ comes with a decomposition $V=V_1\oplus V_2$, and endow the polynomial ring $S=\Sym(V)$ with the bi-grading where $S_{1,0}=V_1$ and $S_{0,1}=V_2$. We regard $V_1\oo V_2$ as the subspace of bi-degree $(1,1)$ elements in $\bw^2 V$, and hence we can regard any subspace $K\subset V_1\oo V_2$ as a subset of $\bw^2 V$ and form the complex \eqref{eq:kosz-sequence}. The resulting homology group $W(V,K)$ is then naturally bi-graded, and we call it a \defi{bi-graded Koszul module}.

We are interested in the case when $\dim(V_i)\geq 2$, where the aforementioned results of \cite{PS} imply that $W(V,K)$ is never of finite length: indeed, if $K \subset V_1 \otimes V_2$ then $\bw^2 V_i^\vee \subset K^\perp$, hence $K^{\perp}$ contains decomposable elements. To generalize the results of \cite{PS}, it is then more convenient to reinterpret the condition that a singly-graded Koszul module $W(V,K)$ has finite length as saying that the associated coherent sheaf on projective space $\bb{P}V$ is $0$. For a bi-graded Koszul module, it is then appropriate to instead consider when the corresponding coherent sheaf on the product of projective spaces $\bb{P}V_1 \times \bb{P}V_2$ is $0$. Algebraically, this means that
\begin{equation}\label{eq:asymp-vanishing}
W_{d,e}(V,K)=0\text{ for }d,e \gg 0.
\end{equation}
Pleasantly, this condition is equivalent to asking that the orthogonal complement $K^\perp \subset (V_1 \otimes V_2)^\vee$ contains no nonzero tensors of rank $\leq 2$ (see Proposition~\ref{prop:sheafW=0}). By picking bases, elements of $(V_1 \otimes V_2)^\vee$ can be interpreted as matrices, and the rank of an element coincides with the usual rank of a matrix. Geometrically, the projectivization of the rank $\le 2$ locus is the secant variety of the Segre embedding of $\bb{P}V_1 \times \bb{P}V_2$, i.e., the Zariski closure of the union of all secant lines through any 2 points of the Segre embedding. Remarkably, in analogy with \cite[Theorem~1.3]{AFPRW} we can make the vanishing \eqref{eq:asymp-vanishing} effective:

\begin{theorem}\label{thm:vanishing}
 Let $p=\chr(\kk)$, $n_i=\dim(V_i)\geq 2$, and suppose that $p=0$ or $p\geq n_1+n_2-3$. We have that
 \[
   W_{d,e}(V,K)=0\mbox{ for }d,e\gg 0 \Longleftrightarrow W_{n_2-2,n_1-2}(V,K)=0.
 \]
\end{theorem}

The condition that $K^{\perp}\subset (V_1 \otimes V_2)^\vee$ contains no nonzero tensors of rank $\leq 2$ can only hold when $\dim(K)\geq 2(n_1+n_2-2)$, which is one more than the dimension of the secant variety of the Segre product $\bb{P}V_1\times\bb{P}V_2$. In analogy with \cite[Theorem~1.4]{AFPRW}, in the borderline case when $\dim(K)=2(n_1+n_2-2)$, we can determine an exact formula for the Hilbert function in low bi-degrees for a module $W(V,K)$ satisfying the conditions in Theorem~\ref{thm:vanishing} (see Theorem~\ref{thm:hilbert-function}). 

\medskip

We now give a high level overview of the strategy of proof of Theorem~\ref{thm:syz-Xab}. If we let $B$ denote the homogeneous coordinate ring of the scroll $\mc{S}(a,b)$, then we have a short exact sequence
\[
  0 \to \omega_B \to A \to B \to 0
\]
where $\omega_B$ is the canonical module of $B$. The minimal free resolution of $B$ is an Eagon--Northcott complex, while the minimal free resolution of $\omega_B$ is obtained by duality. In particular, we have $\Tor_i(B,\kk)_{i+2}=0$, so to prove the desired vanishing of the Tor groups of $A$, we need to show, for $i<\min(a,b)$, the surjectivity of the connecting homomorphisms
\[
  \Tor_{i+1}(B,\kk)_{i+2} \to \Tor_{i}(\omega_B, \kk)_{i+2}.
\]

To write everything invariantly, we pick a 2-dimensional vector space $U$ and write $\bb{P}(U)$ for the corresponding projective line. To keep everything correct in general, we will be careful to distinguish between divided powers $\D$ and symmetric powers $\Sym$; if $\kk$ has characteristic $0$, then these are isomorphic to one another, so the reader may replace all instances of divided powers $\D$ with symmetric powers $\Sym$ if that is their main scenario of interest. The map above takes the form
\begin{equation}\label{eq:invariant-Tor-map}
  \D^i U \otimes \bigwedge^{i+2}(\Sym^{a-1} U \oplus \Sym^{b-1} U) \to \Sym^{a+b-2-i} U \otimes \bigwedge^i(\Sym^{a-1} U \oplus \Sym^{b-1} U)
\end{equation}
where $\D$ is the divided power, $\Sym$ is the symmetric power, and $\bigwedge$ is the exterior power. While it is possible to give explicit formulas for this map, proving surjectivity from such a formula is a difficult task (especially since it depends on the characteristic of the field). Instead, we take a roundabout method that begins with Hermite reciprocity, which is an $\SL(U)$-equivariant isomorphism
\[
  \Sym^d(\D^i U) = \bw^i(\Sym^{d+i-1} U),  
\]
described in \cite[Section~3.4]{AFPRW}. If we decompose both sides of (\ref{eq:invariant-Tor-map}) using the identity
\begin{equation}\label{eq:wedge-sum}
  \bigwedge^d(E \oplus F) = \bigoplus_{u+v=d} \bigwedge^u E \otimes \bigwedge^v F,
\end{equation}
then via Hermite reciprocity, the connecting homomorphism takes the form
  \[
\xymatrix{  
  \displaystyle  \bigoplus_{\substack{u+v=i\\ u,v \ge -1}} \D^i U \otimes \Sym^{a-u-1}(\D^{u+1} U) \otimes \Sym^{b-v-1} (\D^{v+1} U) \ar[d] \\
\displaystyle  \bigoplus_{\substack{u+v=i\\ u,v \ge 0}} \Sym^{a+b-2-i} U \otimes \Sym^{a-u} (\D^u U) \otimes \Sym^{b-v} (\D^v U).
}
\]
If we focus on a particular bi-degree $(u,v)$ and sum over all $a \ge u$ and $b \ge v$, then the domain becomes $\D^i U \otimes \Sym(\D^{u+1} U \oplus \D^{v+1} U)$, a free module over the bi-graded polynomial ring 
\[
  S=\Sym(\D^{u+1} U \oplus \D^{v+1} U),\mbox{ where }S_{1,0}=\D^{u+1} U\mbox{ and }S_{0,1}=\D^{v+1} U.
\]
Miraculously, the target can be given the structure of a finitely generated $S$-module, so that this map is a module homomorphism, namely it is the middle homology of a complex
\[ \D^{u+v+2} U \oo S \lra (\D^{u+1} U \oplus \D^{v+1} U) \oo S \lra S.\]
This identification is subtle and occupies a great deal of the paper! Since $i=u+v$, this leads to a three-term complex of free $S$-modules
\[
  K \oo S \lra (\D^{u+1} U \oplus \D^{v+1} U) \oo S \lra S,
\]
where $K$ is some extension of $\D^{u+v}U$ by $\D^{u+v+2}U$ (which is split if the characteristic of $\kk$ is zero or large). We denote the middle homology by $W^{(u+1,v+1)}$ and call it a \defi{bi-graded Weyman module} (see Section~\ref{sec:Weyman}, and \cite[Section~5.1]{AFPRW} for the singly-graded case). In fact, this is an instance of a bi-graded Koszul module with $V_1=\D^{u+1} U$ and $V_2=\D^{v+1} U$. Specializing Theorem~\ref{thm:vanishing} to this situation gives the following theorem, which itself implies Theorem~\ref{thm:syz-Xab}:

\begin{theorem}\label{thm:vanishing-Weyman}
  If $p=\chr(\kk)$ satisfies $p=0$ or $p>u+v$ then 
  \addtocounter{equation}{-1}
  \begin{subequations}
    \begin{equation*}\label{eq:vanishing-Weyman}
  W^{(u,v)}_{d,e} = 0\mbox{ for }d\geq v,\ e\geq u.
 \end{equation*}
\end{subequations}
\end{theorem}

Finally, we note that in this situation, we have $\dim(K) = 2(\dim(V_1)+\dim(V_2)-2)$; from the previous discussion, we have a formula for the Hilbert function of $W^{(u,v)}$. Based on this, the reader can deduce formulas for certain bi-graded components of the $\Tor$-modules of $A$.

\subsection*{Organization.} In Section~\ref{sec:prelim} we recall basic constructions in multilinear algebra, and discuss Hermite reciprocity. Section~\ref{sec:koszul} is concerned with the basic theory of bi-graded Koszul modules, and contains the proof of the vanishing Theorem~\ref{thm:vanishing}. In Section~\ref{sec:Weyman} we discuss Weyman modules, showing that they satisfy the hypothesis of the vanishing theorem and deriving Theorem~\ref{thm:vanishing-Weyman}. The relationship between the syzygies of K3 carpets and Weyman modules is presented in Section~\ref{sec:K3carpets}, while the geometric applications are summarized in Section~\ref{sec:green}.

\section{Preliminaries}\label{sec:prelim}

In this section we collect some basic facts and notation regarding multilinear algebra, and recall some useful aspects of Hermite reciprocity following \cite{AFPRW}.

\subsection{Multilinear algebra}

Let $E$ be a vector space. The tensor power $E^{\otimes d}$ has an action of the symmetric group $\mf{S}_d$ via permuting tensor factors. The divided power $\D^d E$ is the invariant subspace and the symmetric power $\Sym^d E$ is the space of coinvariants. In formulas:
\begin{align*}
  \D^d E &= \{x \in E^{\otimes d} \mid \sigma(x) = x \ \text{for all $\sigma \in \mf{S}_d$}\}\\
  \Sym^d E &= E^{\otimes d} / \{x - \sigma(x) \mid \sigma \in \mf{S}_d, \  x \in U^{\otimes d}\}.
\end{align*}
There is a natural isomorphism
\[
  (\D^d E)^\vee = \Sym^d(E^\vee).
\]
If $d!$ is nonzero in $\kk$, and in particular if $\chr(\kk)=0$, then the composition $\D^d E \to E^{\otimes d} \to \Sym^d E$ is an isomorphism with inverse $\Sym^d E \cong \D^d E$ given by $x \mapsto \frac{1}{d!} \sum_{\sigma \in \mf{S}_d} \sigma(x)$.

The exterior powers $\bw^d E$ are the skew-invariants of $E^{\otimes d}$, i.e.,
\[
  \bw^d E = \{x \in E^{\otimes d} \mid \sigma(x) = \mathrm{sgn}(\sigma) x \ \text{for all $\sigma \in \mf{S}_d$}\}.
\]
For $e_1,\dots,e_d \in E$, we use the notation
\[
  e_1 \wedge \cdots \wedge e_d = \sum_{\sigma \in \mf{S}_d} \mathrm{sgn}(\sigma) e_{\sigma(1)} \otimes \cdots \otimes e_{\sigma(d)} \in \bw^d E
\]
and $e_1\cdots e_d$ to denote the image of $e_1 \otimes \cdots \otimes e_d$ in $\Sym^d E$.

For $\D^d E$, and $d_1+\cdots+d_r=d$, we use $e_1^{(d_1)} \cdots e_r^{(d_r)}$ to denote the sum over the orbit of $e_1^{\otimes d_1} \otimes \cdots \otimes e_r^{\otimes d_r}$ in $E^{\otimes d}$. For $u,v \ge 0$, we define comultiplication maps
\begin{align*}
  \Delta_{u,v} \colon \D^{u+v} E \to \D^u E \otimes \D^v E
\end{align*}
which are the linear duals of the multiplication maps
\[
  \Sym^u(E^\vee) \otimes \Sym^v(E^\vee) \to \Sym^{u+v}(E^\vee).
\]
Since multiplication is associative, comultiplication is coassociative, i.e., we have $(1 \otimes \Delta_{v,w}) \circ \Delta_{u,v+w} = (\Delta_{u,v} \otimes 1) \circ \Delta_{u+v,w}$ as maps $\D^{u+v+w} E \to \D^u E \otimes \D^v E \otimes \D^w E$.

Similarly, we also define comultiplication maps
\[
  \Delta_{u,v} \colon \bw^{u+v} E \to \bw^u E \otimes \bw^v E
\]
as the linear duals of the multiplication maps
\[
  \bw^u(E^\vee) \otimes \bw^v(E^\vee) \to \bw^{u+v}(E^\vee).
\]
Again, this comultiplication is coassociative. 

\subsection{Hermite reciprocity}

We let $U$ be a 2-dimensional $\kk$-vector space, and use $\SL(U)$ to denote the group of linear operators on $U$ with determinant 1. We fix a basis $\{1,x\}$ for $U$ which gives an identification $\bw^2 U \simeq \kk$ via $1 \wedge x \mapsto 1$, and we use this to identify $U\simeq U^{\vee}$. Hermite reciprocity is an $\SL(U)$-equivariant isomorphism
\[
  \Sym^d(\D^i U) = \bw^i(\Sym^{d+i-1} U).
\]
We won't make use of the explicit form of this isomorphism, but the reader can see \cite[\S 3.4]{AFPRW} for details. Under Hermite reciprocity, the multiplication map
\[
  \D^d U \otimes \Sym^{e-d+1}(\D^d U) \to \Sym^{e-d+2}(\D^d U)
\]
takes the form
\[
  \nu \colon \D^d U \otimes \bw^d(\Sym^e U) \to \bw^d(\Sym^{e+1}U).
\]
See \cite[Eqn. (43) and Proof of Lemma 3.3]{AFPRW} for a formula for $\nu$.

\begin{proposition} \label{prop:comult-nu}
  The following square commutes:
  \[
    \xymatrix{
      \D^d U \oo \bw^d(\Sym^e U) \ar[r]^\nu \ar[d] & \bw^d(\Sym^{e+1} U) \ar[d] \\
      \D^{d-1} U \oo \bw^{d-1}(\Sym^e U) \oo \Sym^{e+1} U \ar[r]^-{\nu \otimes 1} &  \bw^{d-1}(\Sym^{e+1} U) \oo \Sym^{e+1} U
    }
  \]
  where the left map is comultiplication on both factors followed by multiplication, and the right map is exterior comultiplication.
\end{proposition}

\begin{proof}
  See \cite[Proposition 5.9]{AFPRW}.
\end{proof}

\section{Bi-graded Koszul modules}\label{sec:koszul}

In this section we generalize the notion of Koszul modules to the bi-graded setting, and study the natural analogue of finite length modules (see \cite{PS, AFPRW2}). We show that these modules satisfy a strong vanishing theorem, and we give a sharp upper bound for their bi-graded Hilbert function (our results parallel \cite[Theorems~1.3,~1.4]{AFPRW}). We let $V_1,V_2$ denote finite dimensional $\kk$-vector spaces, and let $V = V_1\oplus V_2$. We write $n_i=\dim(V_i)$, assume that $n_i\geq 2$, and let $n=n_1+n_2$. We consider a subspace $K\subseteq V_1\oo V_2$ and let $m=\dim(K)$. We have a decomposition
\[
  \bw^2 V = \bw^2 V_1 \oplus (V_1\oo V_2) \oplus \bw^2 V_2,
\]
which allows us to think of $K$ as a subspace of $\bw^2 V$. We consider the symmetric algebra $S=\Sym(V)$ and define the \defi{Koszul module $W(V,K)$} to be the middle homology of the $3$-term complex
\begin{equation}\label{eqn:W}
\xymatrixcolsep{5pc}
\xymatrix{
K \oo S \ar[r]^{\delta_2|_{K \oo S}} & V\oo S \ar[r]^{\delta_1} & S,
}
\end{equation}
where $\delta_1,\delta_2$ are Koszul differentials.

We consider $S$ as a bi-graded polynomial ring where the elements of $V_1$ have degree $(1,0)$, and those of $V_2$ have degree $(0,1)$. The bi-degree $(d,e)$ component is
\[
  S_{d,e} =  \Sym^d(V_1) \oo \Sym^e(V_2).
\]
The Koszul module $W(V,K)$ inherits a natural bi-grading, where the bi-degree $(d,e)$ component is the homology of
\[
  K \oo S_{d,e} \lra V_1\oo S_{d,e+1} \oplus V_2\oo S_{d+1,e} \lra S_{d+1,e+1}.
\]
We are interested in understanding the vanishing behavior of $W_{d,e}(V,K)$. We note that $W(V,K)$ is generated in bi-degree $(0,0)$, so if $W_{d_0,e_0}(V,K)=0$ for some $(d_0,e_0)$ then $W_{d,e}(V,K)=0$ for all $(d,e)$ with $d\geq d_0$, $e\geq e_0$. 

\begin{theorem}\label{thm:vanishing-koszul}
 Let $p=\chr(\kk)$ and suppose that $p=0$ or $p\geq n-3$. We have that
 \[
   W_{d,e}(V,K)=0\mbox{ for }d,e\gg 0 \Longleftrightarrow W_{n_2-2,n_1-2}(V,K)=0.
 \]
\end{theorem}

As explained in Remark~\ref{rem:Kperp-disj-sec} below, the equivalent conditions in Theorem~\ref{thm:vanishing-koszul} can only be true when $m\geq 2n-4$. If we further assume that $m=2n-4$ then we get an exact formula for the Hilbert function of $W(V,K)$ in low bi-degrees, as follows (compare with \cite[Theorem~1.4]{AFPRW}).

\begin{theorem}\label{thm:hilbert-function}
 With the assumptions in Theorem~\ref{thm:vanishing-koszul}, suppose that $W_{n_2-2,n_1-2}(V,K)=0$. If we let $\Delta_1 = n_1-2-e$ and $\Delta_2=n_2-2-d$, then we have for all $d\leq n_2-2$ and $e\leq n_1-2$ that
 \[
   \dim(W_{d,e}(V,K)) \leq 2 \cdot {d+n_1-1 \choose d}\cdot{e+n_2-1\choose e} \cdot \frac{{n_1-1\choose 2}\cdot \Delta_2 + {n_2-1\choose 2}\cdot \Delta_1 - (n-3)\cdot\Delta_1\cdot\Delta_2}{(d+1)\cdot(e+1)}.
 \]
Moreover, equality holds when $m=2n-4$.
\end{theorem}

To understand geometrically the asymptotic vanishing property of the bi-graded components of $W(V,K)$, we consider the associated Koszul sheaf on $\PP = \bb{P}V_1 \times \bb{P}V_2$, denoted $\mc{W}(V,K)$, and defined as the middle homology of
\begin{subequations}
\begin{equation}\label{eq:def-sheaf-W}
 K \oo \mc{O}_{\PP} \overset{\a}{\lra} V_1 \oo \mc{O}_{\PP}(0,1) \oplus V_2 \oo \mc{O}_{\PP}(1,0) \overset{\beta}{\lra} \mc{O}_{\PP}(1,1).
\end{equation}

In what follows, we let $\mc{G} = \ker(\beta)$, so it fits into the short exact sequence
\begin{equation}\label{eq:def-ses-G}
 0 \lra \mc{G} \lra V_1 \oo \mc{O}_{\PP}(0,1) \oplus V_2 \oo \mc{O}_{\PP}(1,0) \lra \mc{O}_{\PP}(1,1) \lra 0.
\end{equation}
Note that $\mc{G}$ is locally free since $\beta$ is surjective.
\end{subequations}

We have that $W_{d,e}(V,K) = \rH^0(\PP, \mc{W}(V,K) \oo \mc{O}_{\PP}(d,e))$ for $d,e\gg 0$, and in particular the vanishing holds asymptotically if and only if $\mc{W}(V,K)$ is the zero sheaf. To characterize this condition, we define the orthogonal complement of $K$ to be
\[
  K^{\perp} = \{ \phi \in V_1^{\vee} \oo V_2^{\vee} : \phi_{|_K} = 0\}
\]
and prove the following.

\begin{proposition}\label{prop:sheafW=0}
 We have that $\mc{W}(V,K)=0$ if and only if $K^{\perp}$ contains no non-zero tensors of rank at most two.
\end{proposition}

\addtocounter{equation}{-1}
\begin{subequations}
\begin{proof}
  The condition $\mc{W}(V,K)=0$ is equivalent to the exactness of (\ref{eq:def-sheaf-W}) in the middle, which in turn is equivalent to the surjectivity of the induced map $\a\colon K \oo \mc{O}_{\PP} \lra \mc{G}$. This can be checked fiber by fiber, and since $\mc{G}$ is locally free, the middle exactness of (\ref{eq:def-sheaf-W}) can also be checked fiber by fiber. Fix a $\kk$-point $p=([f_1],[f_2]) \in \PP$, with $f_i\in V_i^{\vee}$, and restrict (\ref{eq:def-sheaf-W}). We get a complex of vector spaces
 \[ K \overset{\a_p}{\lra} V_1 \oplus V_2 \overset{f_1\oplus f_2}{\lra} \kk,\]
 which is exact if and only if the dual complex
 \[ \kk \overset{(f_1,f_2)}{\lra} V_1^{\vee} \oplus V_2^{\vee} \overset{\a_p^{\vee}}{\lra} K^{\vee}\]
 is exact. Writing $V^{\vee} = V_1^{\vee} \oplus V_2^{\vee}$ and $f=(f_1,f_2)$, we observe that the map $\a_p^{\vee}$ is obtained as the composition
 \begin{equation}\label{eq:dual-fiber-sheaf-W}
  V^{\vee} \overset{\wedge f}{\lra} \bw^2 V^{\vee} \onto K^{\vee},
 \end{equation}
 where the second map is the dual projection to the inclusion $K\subset\bw^2 V$, and therefore has kernel equal to $\bw^2 V_1^{\vee} \oplus K^{\perp} \oplus \bw^2 V_2^{\vee}$. It follows that (\ref{eq:dual-fiber-sheaf-W}) fails to be exact if and only if one can find $g=(g_1,g_2)\in V^{\vee}$ which is not a multiple of $f$ and such that $f\wedge g\in \bw^2 V_1^{\vee} \oplus K^{\perp} \oplus \bw^2 V_2^{\vee}$. Since
 \[
   f\wedge g = (f_1\wedge g_1, f_1\oo g_2 - g_1 \oo f_2,f_2\wedge g_2),
 \]
 we get that (\ref{eq:dual-fiber-sheaf-W}) fails to be exact if and only if $K^{\perp}$ contains a non-zero tensor $f_1\oo g_2 - g_1 \oo f_2$ of rank at most two.
\end{proof}
\end{subequations}

\begin{remark}\label{rem:Kperp-disj-sec}
Note that $K^{\perp}$ defines a linear space $H$ of codimension $m$ in $\bb{P}(V_1\oo V_2)$, which in turn is the ambient space of the Segre embedding $X$ of $\bb{P}V_1 \times \bb{P}V_2$. The condition in Proposition~\ref{prop:sheafW=0} is then equivalent to the fact that $H$ is disjoint from $\op{Sec}(X)$, the variety of secant lines to $X$. Since $\dim(\op{Sec}(X))=2n-5$, this is only possible when $m\geq 2n-4$. Moreover, if $H$ is generic of codimension $m=2n-4$ then $H\cap \op{Sec}(X) = \emptyset$.
\end{remark}

\begin{lemma}\label{lem:vanish-Sym-Gvee}
 For $r=0,\dots,n-4$ we have that $\Sym^r(\mc{G}^{\vee}) \oo \mc{O}_{\bf P}(-1,-1)$ has no non-zero cohomology groups.
\end{lemma}

\begin{proof}
 Dualizing (\ref{eq:def-ses-G}) and taking symmetric powers, we get a short exact sequence
 \[
   0 \lra \Sym^{r-1}(\mc{V}) \oo \mc{O}_{\PP}(-2,-2) \lra \Sym^r(\mc{V}) \oo \mc{O}_{\PP}(-1,-1) \lra \Sym^r(\mc{G}^{\vee}) \oo \mc{O}_{\PP}(-1,-1) \lra 0,
 \]
 where
 \[
   \mc{V} = V_1^{\vee} \oo \mc{O}_{\PP}(0,-1) \oplus V_2^{\vee} \oo \mc{O}_{\PP}(-1,0).
 \]
 It is then enough to check that the sheaves $\Sym^r(\mc{V}) \oo \mc{O}_{\PP}(-1,-1)$ and $\Sym^{r-1}(\mc{V}) \oo \mc{O}_{\PP}(-2,-2)$ have no non-zero cohomology groups.

 First note that $\Sym^r(\mc{V}) \oo \mc{O}_{\PP}(-1,-1)$ decomposes as a direct sum of $\mc{O}_{\PP}(i,j)$ with $i,j<0$ and $i+j = -r-2 \geq -n+2$, while $\Sym^{r-1}(\mc{V}) \oo \mc{O}_{\PP}(-2,-2)$ decomposes as a direct sum of $\mc{O}_{\PP}(i,j)$ with $i,j<0$ and $i+j = (-r+1)-4 \geq -n+1$.

Next, the condition $i+j\geq -n+1$ implies that either $i\geq -n_1+1$ or $j\geq -n_2+1$, so at least one of $\mc{O}_{\bb{P}V_1}(i)$ or $\mc{O}_{\bb{P}V_2}(j)$ has no non-zero cohomology groups and so $\mc{O}_{\PP}(i,j)$ has no non-zero cohomology groups by K\"unneth's formula.
\end{proof}

In the next proof we will need the Buchsbaum--Rim complex, for which we recall the important details now. Let $X$ be a scheme and $\alpha \colon E \to F$ be a morphism of locally free sheaves on $X$ with $\rm{rank}(E)=e$ and $\rm{rank}(F)=f$ (and $e \ge f$). The Buchsbaum--Rim complex $B(\alpha)_\bullet$ of $\alpha$ has terms
\begin{align*}
  B(\alpha)_0 &= F,\\
  B(\alpha)_1 &= E,\\
  B(\alpha)_r &= \bw^{r+f-1} E \otimes \det (F^\vee) \otimes \D^{r-2} (F^\vee) \mbox{ for }r=2,\dots,e-f+1,
\end{align*}
and the differential $B(\alpha)_1 \to B(\alpha)_0$ is $\alpha$ (this is the complex $\mc{C}^1$ in \cite[\S A2.6]{eisenbud-book} where it is treated in the local setting -- the terms $\det(F^\vee)$, which are omitted there are necessary to globalize this construction). This is exact in positive degrees if the ideal sheaf of maximal minors of $\alpha$ has depth $\ge e-f+1$, with the convention that the unit ideal has infinite depth (exactness can be checked locally, in which case it follows from \cite[Theorem A2.10]{eisenbud-book}). In our application below, $X = \PP$ and $\alpha$ is surjective, and hence the ideal sheaf of maximal minors is the unit ideal.

\begin{proof}[Proof of Theorem~\ref{thm:vanishing-koszul}]
 The implication ``$\Leftarrow$" follows from the fact that $W(V,K)$ is generated in bi-degree $(0,0)$. To prove ``$\Rightarrow$", we first reduce to the case $m=2n-4$. As remarked earlier, the vanishing $W_{d,e}(V,K)=0$ for $d,e\gg 0$ is equivalent to $\mc{W}(V,K) = 0$, which is further equivalent to the fact that the linear space $H\subseteq \bb{P}(V_1\oo V_2)$ corresponding to $K^{\perp}$ is disjoint from $\op{Sec}(X)$. Assuming that this condition is satisfied, a generic choice of a linear space $H'\supseteq H$ of codimension $2n-4$ will still have the property that $H'\cap \op{Sec}(X) = \emptyset$, so it gives a bi-graded Koszul module $W(V,K')$ with $K'\subseteq K$ and $\dim(K')=2n-4$. The inclusion $K'\subseteq K$ induces a natural surjection $W(V,K') \onto W(V,K)$, so a vanishing for $W(V,K')$ will imply the corresponding vanishing for $W(V,K)$.
 
 We assume that $m=2n-4$ and let $\mc{G} = \ker(\beta)$ as in the proof of Proposition~\ref{prop:sheafW=0}. Since $\mc{W}(V,K)=0$, we have that the map $\a\colon K\oo\mc{O}_{\PP}\lra\mc{G}$ is surjective, so it gives an exact Buchsbaum--Rim complex  $\mc{B}_{\bullet}$ with 
 \begin{align*}
   \mc{B}_0&=\mc{G},\\
   \mc{B}_1 &= K\oo\mc{O}_{\PP},\\
   \mc{B}_r &= \bw^{n+r-2}K \oo \det\left(\mc{G}^{\vee}\right) \oo \D^{r-2}\left(\mc{G}^{\vee}\right)\mbox{ for }r=2,\dots,n-2
 \end{align*}
  The condition $W_{n_2-2,n_1-2}(V,K)=0$ is equivalent to the fact that after twisting by $\mc{O}_{\PP}(n_2-2,n_1-2)$, the induced map on global sections
\begin{equation}\label{eq:twist-H0-B}
\rH^0(\PP,\mc{B}_1(n_2-2,n_1-2)) \lra \rH^0(\PP,\mc{B}_0(n_2-2,n_1-2))
\end{equation}
is surjective. Since $\mc{B}_{\bullet}(n_2-2,n_1-2)$ is an exact complex, its hypercohomology groups are all zero. Using the hypercohomology spectral sequence, in order to prove the surjectivity of (\ref{eq:twist-H0-B}) it suffices to check that the sheaves $\mc{B}_{r}(n_2-2,n_1-2)$ have no cohomology (in fact, it is enough that $\rH^{r-1}(\PP,\mc{B}_{r}(n_2-2,n_1-2))=0$) for $r=2,\cdots,n-2$.
  
Since $0\leq r-2\leq n-4$, it follows from our hypothesis that $p=0$ or $p>r-2$, thus $\D^{r-2}(\mc{G}^{\vee}) = \Sym^{r-2}(\mc{G}^{\vee})$. Moreover, we have that $\det(\mc{G}^{\vee}) = \mc{O}_{\PP}(-n_2+1,-n_1+1)$, so
\[
  \mc{B}_{r}(n_2-2,n_1-2) = \bw^{n+r-2}K \oo \mc{O}_{\PP}(-1,-1) \oo \Sym^{r-2}\left(\mc{G}^{\vee}\right),\mbox{ for }r=2,\dots,n-2.
\]
The desired vanishing now follows from Lemma~\ref{lem:vanish-Sym-Gvee}.
\end{proof}

\begin{proof}[Proof of Theorem~\ref{thm:hilbert-function}]
 Using the projection argument from the proof of Theorem~\ref{thm:vanishing-koszul} it suffices to consider the case when $m=2n-4$ and show that we get an exact formula for $\dim(W_{d,e}(V,K))$ in the given range. Restricting (\ref{eqn:W}) to bi-degree $(d,e)$, we get a complex
\[ K\oo S_{d,e} \overset{\a_{d,e}}{\lra} V_1\oo S_{d,e+1} \oplus V_2\oo S_{d+1,e}  \overset{\b_{d,e}}{\lra}  S_{d+1,e+1}\]
whose middle homology is $W_{d,e}(V,K)$. We get that $\dim(W_{d,e}(V,K)) \geq \chi_{d,e}$, where
\[ \chi_{d,e} = \dim\left(V_1\oo S_{d,e+1} \oplus V_2\oo S_{d+1,e}\right) - \dim(S_{d+1,e+1}) - \dim(K\oo S_{d,e})\]
is the Euler characteristic of the above complex. Moreover, since $\b_{d,e}$ is surjective, we have that $\dim(W_{d,e}(V,K)) = \chi_{d,e}$ if and only if $\a_{d,e}$ is injective. A direct calculation shows that
\[ \chi_{d,e} = 2\cdot{d+n_1-1 \choose d}\cdot{e+n_2-1\choose e} \cdot \frac{{n_1-1\choose 2}\cdot \Delta_2 + {n_2-1\choose 2}\cdot \Delta_1 - (n-3)\cdot\Delta_1\cdot\Delta_2}{(d+1)\cdot(e+1)},\]
so to prove Theorem~\ref{thm:hilbert-function} it suffices to show that $\a_{d,e}$ is injective for $d\leq n_2-2$, $e\leq n_1-2$. Since $\a_{d,e}$ is a homogeneous component of a map of free modules, we have that if $\a_{d_0,e_0}$ is injective then $\a_{d,e}$ is also injective for all $d\leq d_0$ and $e\leq e_0$. It is then enough to prove that $\a_{n_2-2,n_1-2}$ is injective. Notice that for $d=n_2-2$ and $e=n_1-2$, we have $\Delta_1=\Delta_2=0$, so $\chi_{n_2-2,n_1-2}=0$. Moreover, we know by Theorem~\ref{thm:vanishing-koszul} that $W_{n_2-2,n_1-2}(V,K)=0$, $\a_{n_2-2,n_1-2}$ is injective.
\end{proof}

\section{Bi-graded Weyman modules}\label{sec:Weyman}

The fundamental connection described in \cite{AFPRW} between (standard graded) Koszul modules and syzygies goes through \defi{Weyman modules}. We define their analogues in the bi-graded setting, and show that they satisfy (in most characteristics) the hypothesis of Theorem~\ref{thm:vanishing-koszul}. 

For $i,j\geq 0$ we consider the surjective multiplication map
\[
  \mu_{u,v} \colon  \Sym^u U \oo \Sym^v U \lra \Sym^{u+v}U.
\]
The kernel of $\mu_{u,v}$ is naturally identified with $\Sym^{u-1} U \oo \Sym^{v-1} U$ via the inclusion
\begin{align*}
  \iota_{u,v} \colon  \Sym^{u-1} U \oo \Sym^{v-1} U &\to \Sym^u U \oo \Sym^v U,\\
  f\oo g &\mapsto f \oo xg - xf\oo g.
\end{align*}
More generally, for $t\leq u,v$ the composition $\iota^t_{u,v} = \iota_{u,v} \circ \iota_{u-1,v-1} \circ \cdots \circ \iota_{u-t+1,v-t+1}$ is given by
\[
  f\oo g \mapsto \sum_{i=0}^t (-1)^i{t\choose i} \cdot x^if\oo x^{t-i}g.
\]

We let
\[ Q_{u,v} = \coker(\iota_{u,v} \circ \iota_{u-1,v-1}),\]
which gives a short exact sequence
\[
  0\lra \Sym^{u-2} U \oo \Sym^{v-2} U \lra \Sym^u U \oo \Sym^v U \overset{\Psi_{u,v}}{\lra} Q_{u,v} \lra 0.
\]
In characteristic zero (or sufficiently large characteristic), one has an $\SL(U)$-equivariant decomposition $Q_{u,v} \simeq \Sym^{u+v}U \oplus  \Sym^{u+v-2}U$, but in general we only have an extension
\begin{equation}\label{eq:ses-Quv}
0\lra \Sym^{u+v-2}U \lra Q_{u,v} \lra \Sym^{u+v}U \lra 0.
\end{equation}

\begin{remark}
 If $\chr(\kk)=p>0$ then ${p \choose i}=0$ in $\kk$ for $0<i<p$, and thus for $u,v\geq p$ we have
 \[
   \iota^p_{u,v}(f\oo g) = f\oo x^p g - x^p f\oo g.
 \]
 Since $\op{Im}(\iota^p_{u,v}) \subset \op{Im}(\iota^2_{u,v}) = \ker(\Psi_{u,v})$, this shows that $\ker(\Psi_{u,v})$ contains rank two tensors. We will show that this is no longer the case when $p>\min(u,v)$.
\end{remark}

We let $V_i = (\Sym^{n_i-1}U)^{\vee} = \D^{n_i-1}U$, and let $K = Q_{n_1-1,n_2-1}^{\vee}$, considered as a subspace of $V_1\oo V_2$ via the inclusion $\Psi_{n_1-1,n_2-1}^{\vee}$. We note that 
\[ \dim(V_i)=n_i\mbox{ and }\dim(K) = 2n-4.\]
We define the \defi{bi-graded Weyman module $W^{(n_1-1,n_2-1)}:= W(V,K)$}.

\begin{proposition}\label{prop:ker-Psi}
 Let $p=\chr(\kk)$ and suppose that $p=0$ or that $p>\min(u,v)$. Then $\ker(\Psi_{u,v})$ contains no non-zero tensors of rank at most two.
\end{proposition}

\begin{proof}
 We assume without loss of generality that $u \le v$, identify as usual $\Sym^d U$ with polynomials of degree $\leq d$ in $x$, and consider the derivation $\partial = \frac{\partial}{\partial x} \colon \kk(x)\lra\kk(x)$. We note that $\ker(\Psi_{u,v}) = \op{Im}(\iota^2_{u,v})$ is generated by elements of the form
 \[ \iota^2_{u,v}(x^a \oo x^b) = x^a \oo x^{b+2} - 2x^{a+1}\oo x^{b+1} + x^{a+2} \oo x^b,\]
 which are both in $\ker(\mu_{u,v})$ and in the kernel of the composition
 \[ 
 \xymatrix{
 \Sym^u U \oo \Sym^v U \ar[rr]^-{\partial \oo \op{id}}  \ar@/_2.0pc/[rrrr]_{\Phi} & &  \Sym^{u-1} U \oo \Sym^v U \ar[rr]^-{\mu_{u-1,v}} & & \Sym^{u+v-1}U \\
 }
 \]
Suppose now that $T\in\ker(\Psi_{u,v})$ is a non-zero tensor of rank at most $2$. If $T=f \oo g$, then $0=\mu_{u,v}(T)=fg$, which forces either $f=0$ or $g=0$, contradicting the fact that $T\neq 0$. We may therefore assume that
\[ T = f_1 \oo g_1 + f_2\oo g_2,\quad f_1,f_2\in\Sym^u U\text{ are not proportional, and }T\in\ker(\mu_{u,v})\cap\ker(\Phi).\]
Using the fact that $0 = \mu_{u,v}(T) = f_1g_1+f_2g_2$ we get
\[ \partial\left(\frac{f_1}{f_2}\right) = \frac{(\partial f_1)f_2 - (\partial f_2)f_1}{f_2^2} =  \left[(\partial f_1)g_1 + (\partial f_2)g_2\right]\frac{1}{f_2g_1} = \frac{\Phi(T)}{f_2g_1} = 0.\]
Since $\ker(\partial)=\kk(x^p)$, we conclude that $\frac{f_1}{f_2}\in\kk(x^p)$. By our hypothesis, we have that $p=0$ or $p>u$, which in turn forces $\frac{f_1}{f_2}\in\kk$, contradicting the fact that $f_1,f_2$ were not proportional.
\end{proof}

It follows from Proposition~\ref{prop:ker-Psi} that if $p=0$ or $p\geq\min(n_1,n_2)$ then Proposition~\ref{prop:sheafW=0} applies to the bi-graded Weyman module $W^{(n_1-1,n_2-1)}$. We get using Theorem~\ref{thm:vanishing-koszul} the following.

\begin{corollary}\label{cor:vanishing-Weyman}
 If $n_1,n_2\geq 2$ and $p=\chr(\kk)$ satisfies $p=0$ or $p\geq n_1+n_2-3$ then 
 \[ W^{(n_1-1,n_2-1)}_{d,e} = 0\mbox{ for }d\geq n_2-2,\ e\geq n_1-2.\]
\end{corollary}

\section{Syzygies of K3 carpets}\label{sec:K3carpets}

Fix positive integers $a,b$. In this section we study the syzygies of the \defi{K3 carpet} $\mc{X}(a,b)$, obtained as a double structure on a rational normal scroll of type $(a,b)$. We show that via Hermite reciprocity, these syzygies can be built from components of bi-graded Weyman modules. Using Corollary~\ref{cor:vanishing-Weyman}, this yields a vanishing result for syzygies of K3 carpets that was conjectured by Eisenbud and Schreyer in \cite{ES-K3carpets}.

\subsection{Rational normal scrolls}

Let
\[
  \mc{S}(a,b)\subseteq\bb{P}(\Sym^a U \oplus \Sym^b U) \simeq \PP^{a+b+1}
\]
denote the rational normal scroll of type $(a,b)$. It is abstractly isomorphic to the projective bundle $\bb{P}_{\bb{P}U}(\mc{E})$, where $\mc{E} = \mc{E}_1 \oplus \mc{E}_2$, $\mc{E}_1=\mc{O}_{\bb{P}U}(a)$, $\mc{E}_2=\mc{O}_{\bb{P}U}(b)$. Let $B$ denote the homogeneous coordinate ring of the scroll, which is naturally bi-graded with
\[
  B_{d,e} = \rH^0(\bb{P}U,\Sym^d(\mc{E}_1)\oo \Sym^e(\mc{E}_2)) = \Sym^{da+eb}U.
\]
We let
\[
  R=\Sym(\Sym^aU \oplus \Sym^b U)
\]
denote the homogeneous coordinate ring of the ambient projective space, with its natural bi-grading. We have that $B=R/I$, where $I$ is the ideal of the scroll, generated by 
\[
  \bw^2 U \oo \bw^2(\Sym^{a-1}U \oplus \Sym^{b-1} U) \subset R_{1,1}.
\]
More explicitly, the multiplication map
\[
  U \otimes (\Sym^{a-1} U \oplus \Sym^{b-1} U) \to \Sym^a U \oplus \Sym^b U
\]
can be represented as a $2 \times (a+b)$ matrix whose entries are the linear forms in $R$, and $I$ is generated by the $2 \times 2$ minors of this matrix. In particular, it is resolved by an Eagon--Northcott complex and so
\begin{align} \label{eqn:scroll-tor}
  \Tor_i(B, \kk)_{i+1} = \D^{i-1} U \otimes \bw^{i+1}(\Sym^{a-1} U \oplus \Sym^{b-1} U) \qquad 1 \le i \le a+b-1.
\end{align}

The canonical module $\omega_B$ of $B$ is identified with $\rH^0(\bb{P}U,\omega_{\bb{P}U} \oo \det(\mc{E}) \oo \Sym(\mc{E}))$, with bi-grading
\[
  (\omega_B)_{d,e} = \rH^0(\bb{P}U,\omega_{\bb{P}U} \oo \det(\mc{E}) \oo \Sym^{d-1}(\mc{E}_1)\oo \Sym^{e-1}(\mc{E}_2)),
\]
and in particular it is generated in bi-degree $(1,1)$ by $(\omega_B)_{1,1} = \Sym^{a+b-2}U$. Dualizing \eqref{eqn:scroll-tor} and taking into account the bi-grading gives (with $u+v=i$)
\begin{align} \label{eqn:canonical-tor}
  \Tor_i^R(\omega_B, \kk)_{u+1,v+1} = \Sym^{a+b-2-u-v} U \oo \bw^u(\Sym^{a-1} U) \oo \bw^v(\Sym^{b-1} U).
\end{align}

We have a surjective map $\phi\colon I \lra \omega_B$, which at the level of generators is given by a map
\[
  \xymatrix{
    I_{2,0} \oplus I_{1,1} \oplus I_{0,2} \ar@{=}[r] & \bw^2(\Sym^{a-1}U) \oplus (\Sym^{a-1}U\oo \Sym^{b-1}U) \oplus \bw^2(\Sym^{b-1}U)\ar[d]^{\phi_{1,1}} \\
    (\omega_B)_{1,1}\ar@{=}[r] & \Sym^{a+b-2}U
    }
\]
where $\phi_{1,1}$ sends $\bw^2(\Sym^{a-1}U)$ and $\bw^2(\Sym^{b-1}U)$ to zero, and it is described on $\Sym^{a-1}U\oo \Sym^{b-1}U$ by the natural multiplication map. We let $A$ denote the coordinate ring of the associated \defi{K3 carpet $\mc{X}(a,b)$}, which is obtained as an $R$-module extension
\begin{equation}\label{eq:ses-omega-A-B}
0 \lra \omega_B \lra A \lra B \lra 0
\end{equation}
induced by $\phi \in \Hom_R(I,\omega_B) = \Ext^1_R(B,\omega_B)$. 

For the next result, we collapse the bi-grading on $A$ to a single grading by $A_n = \bigoplus_{i+j=n} A_{i,j}$.

\begin{proposition} \label{prop:A-hilbert}
  The Hilbert series of $A$ is
  \[
    \sum_{n \ge 0} (\dim A_n) t^n = \frac{1 + (a+b-1)t + (a+b-1)t^2 + t^3}{(1-t)^3}.
  \]  
\end{proposition}

\begin{proof}
  The Hilbert series of $A$ is a sum of the Hilbert series of $B$ and $\omega_B$, so we calculate each separately. We have $B_{i,j} = \Sym^{ia+jb} U$, so $\dim B_{i,j} = ia+jb+1$, and hence
  \[
    \dim B_n = \sum_{i=0}^n (ia+(n-i)b+1) = \frac{n(n+1)}{2}(a+b) + n+1.
  \]
  Next, we have $(\omega_B)_{i,j} = \Sym^{ia+jb-2} U$ for $i,j \ge 1$ and 0 otherwise, and in particular,
  \[
    \dim (\omega_B)_n = \sum_{i=1}^{n-1} (ia+(n-i)b-1) = \frac{n(n-1)}{2}(a+b) - n+1 \qquad (n \ge 1).
  \]
  So
  \[
    \dim A_0 = 1, \qquad \dim A_n = n^2(a+b) + 2 \qquad (n \ge 1),
  \]
  and
  \[
    \sum_{n \ge 0} (\dim A_n) t^n = \frac{(a+b)t(t+1)}{(1-t)^3} + \frac{2}{1-t} - 1 = \frac{1+(a+b-1)t + (a+b-1)t^2 + t^3}{(1-t)^3}. \qedhere
  \]
\end{proof}

\subsection{The main result}

One has that $A$ is Gorenstein with Castelnuovo--Mumford regularity $3$, and it is conjectured in \cite{ES-K3carpets} that 
\[
  \Tor_i^R(A,\kk)_{i+2}=0\mbox{ for }i<\min(a,b),
\]
provided that $p=\chr(\kk)$ satisfies $p=0$ or $p\geq\min(a,b)$. We prove this conjecture as a consequence of our basic results on bi-graded Koszul modules. More precisely, we show the following.

\begin{theorem}\label{thm:Ki2=Weyman}
 Consider non-negative integers $u,v\geq 0$ with $u+v=i$. We have that 
 \[
   \Tor_i^R(A,\kk)_{u+1,v+1} \simeq W^{(u+1,v+1)}_{a-1-u,b-1-v}.
 \]
 In particular, if $p=\chr(\kk)$ satisfies $p=0$ or $p\geq\min(a,b)$ and if $i<\min(a,b)$ then 
 \[
   \Tor_i^R(A,\kk)_{i+2}=0.
 \]
\end{theorem}

To prove the first part of the theorem, we note that (\ref{eq:ses-omega-A-B}) induces an exact sequence
\[\cdots \lra \Tor_{i+1}^R(B,\kk)_{i+2} \lra \Tor_i^R(\omega_B,\kk)_{i+2} \lra \Tor_i^R(A,\kk)_{i+2} \lra \Tor_i^R(B,\kk)_{i+2} \lra \cdots\]
Since $\Tor_i^R(B,\kk)_{i+2}=0$ for all $i$, and $\Tor_{i+1}^R(B,\kk)_{i+2}=\Tor_i^R(I,\kk)_{i+2}$, it follows that
\[\Tor_i^R(A,\kk)_{i+2} = \coker\left(\Tor_i^R(I,\kk)_{i+2} \lra \Tor_i^R(\omega_B,\kk)_{i+2}\right),\]
where the maps are induced by the surjection $I\onto \omega_B$ described earlier.

  \begin{proposition} \label{prop:scroll-free}
    We have
    \[
      \Tor_i^R(I,\kk)_{u+1,v+1} = \D^{u+v}U \oo \Sym^{a-1-u}(\D^{u+1}U) \oo \Sym^{b-1-v}(\D^{v+1}U).
      \]
  \end{proposition}

\begin{proof}
Using \eqref{eqn:scroll-tor}, we have
\[
  \Tor_i^R(I,\kk)_{u+1,v+1} = \D^{u+v}U \oo \bw^{u+1}(\Sym^{a-1}U) \oo \bw^{v+1}(\Sym^{b-1}U)
\]
so the identification follows abstractly from Hermite reciprocity.
\end{proof}

Define
\begin{align*}
  \ms{S}(u,v) &= \Sym(\D^{u+1} U \oplus \D^{v+1} U)\\ 
  \ms{M}(u,v) &= \bigoplus_{d+e \ge 2} \Sym^{d+e-2} U \otimes \Sym^d (\D^uU) \otimes \Sym^e(\D^vU).
\end{align*}
For simplicity, we will also write $S$ for $\ms{S}(u,v)$.
Both have bi-gradings via:
\begin{align*}
  S_{d,e} &= \Sym^d(\D^{u+1} U) \oo \Sym^e(\D^{v+1} U)\\
  \ms{M}(u,v)_{d,e} &= \Sym^{d+e-2} U \oo \Sym^d(\D^uU) \oo \Sym^e(\D^vU).
\end{align*}
We will see in the proof of the next result that $\ms{M}(u,v)$ can be given the structure of a finitely generated $S$-module.

\begin{proposition} \label{prop:omega-koszul}
$\Tor_i^R(\omega_B,\kk)_{u+1,v+1} = W_{a-1-u,b-1-v}(V,K)$, where $V = \D^{u+1}U\oplus \D^{v+1}U$, and $K=\D^{u+v+2}U$.
\end{proposition}

\begin{proof}
Apply Hermite reciprocity to \eqref{eqn:canonical-tor} to get
  \[
    \Tor_i^R(\omega_B, \kk)_{u+1,v+1} = \Sym^{a+b-2-u-v} U \oo \Sym^{a-u}(\D^u U) \oo \Sym^{b-v}(\D^vU).
  \]
We have a short exact sequence of vector bundles over $\bb{P}(U)$:
\[
  0 \to  \mc{O}(-u-1) \oplus \mc{O}(-v-1) \to \D^{u+1} U \oplus \D^{v+1} U \to (\D^u U)(1) \oplus (\D^v U)(1) \to 0
\]
Using \cite[\S 5]{weyman}, we have a minimal complex ${\bf F}_\bullet$ with terms
\[
  {\bf F}_i = \bigoplus_{j \ge 0} \rH^j(\bb{P}U, \bw^{i+j}( \mc{O}(-u-1) \oplus \mc{O}(-v-1)) \oo \mc{O}(-2)) \otimes S(-i-j)
\]
whose homology is 
\begin{align*}
  \rH_0({\bf F}_\bullet) &= \rH^0(\bb{P}(U); \Sym((\D^u U)(1) \oplus (\D^v U)(1)) \otimes \mc{O}(-2)) = \ms{M}(u,v)\\
  \rH_{-1}({\bf F}_\bullet) &= \rH^1(\bb{P}(U); \Sym((\D^u U)(1) \oplus (\D^v U)(1)) \otimes \mc{O}(-2)) = \kk.
\end{align*}
Here we treat the terms as singly-graded modules, though it can be made bi-graded by setting $\deg(\D^{u+1}U)=(1,0)$ and $\deg(\D^{v+1}U)=(0,1)$. Explicitly, the terms are
\begin{align*}
  {\bf F}_{-1} &= S\\
  {\bf F}_0 &= (\D^{u+1} U \otimes S(-1,0)) \oplus (\D^{v+1} U \otimes S(0,-1))\\
  {\bf F}_1 &= \D^{u+v+2} U \otimes S(-1,-1).
\end{align*}
Hence $\ms{M}(u,v)$ is realized as a bi-graded Koszul module with $K = \D^{u+v+2} U$, $V_1 = \D^{u+1} U$, and $V_2= \D^{v+1} U$ and so
\[
  \Tor_i^R(\omega_B, \kk)_{u+1,v+1} = \ms{M}(u,v)_{a-u, b-v} = W_{a-1-u, b-1-v}(V,K). \qedhere
\]
\end{proof}

Using the dual of (\ref{eq:ses-Quv}), one can form the Weyman module $W^{(u+1,v+1)} = W(V,Q_{u+1,v+1}^{\vee})$ in two steps: we first use the subspace $K=\D^{u+v+2}U \subset Q_{u+1,v+1}^{\vee}$ and form the Koszul module $W(V,K)$ in part (2). Then there is a natural map 
 \begin{equation}\label{eq:map-Duv-WVK}
 \D^{u+v}U \oo S(-1,-1) \lra W(V,K)
 \end{equation}
 induced by the identification $\D^{u+v}U \simeq Q_{u+1,v+1}^{\vee} / K$, and the cokernel of this map is by definition $W(V,Q_{u+1,v+1}^{\vee}) = W^{(u+1,v+1)}$.

\subsection{Some complexes}

 Make the following definitions:
 \begin{align*}
   \ms{A}(u,v) &= \D^{u+v+2} U \oo S(-1,-1) & (u,v \ge -1)\\
   \ms{B}(u,v) &= \D^{u+1} U \oo \D^{v+1} U \oo S(-1,-1) & (u,v \ge -1)\\
   \ms{C}'(u,v) &= \D^u U \oo \D^v U \oo S(-1,-1) & (u,v \ge 0)\\
   \ms{C}(u,v) &= \D^{u-1} U \oo \D^{v-1} U \oo S(-1,-1) & (u,v \ge 1)\\
   \ms{D}_1(u,v) &= \D^{u+1} U \oo S(-1,0) & (u \ge -1)\\
   \ms{D}_2(u,v) &= \D^{v+1} U \oo S(0,-1) & (v \ge -1)\\
   \ms{D}(u,v) &= \ms{D}_1(u,v) \oplus \ms{D}_2(u,v) & (u,v \ge -1)\\
   \ms{N}(u,v) &= \D^{u+v} U \oo S(-1,-1) & (u,v \ge 0).
 \end{align*}
 
 The map \eqref{eq:map-Duv-WVK} is the middle homology of the following map between 3-term complexes:
 \begin{align} \label{eqn:complexes}
   \xymatrix{
     \ms{C}(u,v) \ar[r]^0 & \ms{S}(u,v) \\
   \ms{B}(u,v) \ar[r] \ar[u]^{(\iota_{u,v} \circ\iota_{u+1,v+1})^*} & \ms{D}(u,v) \ar[u] \\
     \ms{A}(u,v) \ar[r]^{\mathrm{id}} \ar[u]^{\mu_{u+1,v+1}^*} & \ms{A}(u,v) \ar[u]
     }
 \end{align}
 The right-hand side is just the complex computing $W(V,K)$ and the middle horizontal map comes from the inclusion
   \[
     \D^{u+1} U \oo \D^{v+1} U \to \bw^2(\D^{u+1} U \oplus \D^{v+1} U) \to (\D^{u+1} U \oplus \D^{v+1} U)^{\otimes 2}.
   \]

   Let $Z$ be one of the symbols $\ms{A}, \ms{B}, \ms{C}', \ms{C}, \ms{D}, \ms{S}$. We construct a double complex\footnote{Our differentials will only be correct up to a sign. Choosing a sign convention is a purely formal matter which we will ignore in favor of readability.} $\Phi(Z)$ of free $R$-modules with terms
   \[
     \Phi(Z)_{u,v} = Z(u,v)_{a-u,b-v} \otimes R.
   \]
   We will now describe the differentials, which on generators take the form
   \begin{align*}
     Z(u,v)_{a-u,b-v} &\to Z(u-1,v)_{a-u+1,b-v} \oo \Sym^a U\\
     Z(u,v)_{a-u,b-v} &\to Z(u,v-1)_{a-u,b-v+1} \oo \Sym^b U.
   \end{align*}
   We call the first map the ``$u$-component'' and the second map the ``$v$-component''.

   For the cases $Z \in \{\ms{A}, \ms{B}, \ms{C}', \ms{C}\}$, we can write $Z(u,v)$ as $G_Z(u,v) \otimes S(-1,-1)$. In each of these cases, we have two maps
   \begin{align*}
     G_Z(u,v) &\to G_Z(u-1,v) \otimes U\\
     G_Z(u,v) &\to G_Z(u,v-1) \otimes U
   \end{align*}
   via comultiplication. We will describe the differentials in terms of these maps for these cases.

   The $u$-component takes the form
   \[
     \xymatrix{
     G_Z(u,v) \oo \Sym^{a-u-1}(\D^{u+1} U) \oo \Sym^{b-v-1}(\D^{v+1} U) \ar[d]\\
     G_Z(u-1,v) \oo \Sym^{a-u}(\D^{u} U) \oo \Sym^{b-v-1}(\D^{v+1} U) \oo \Sym^a U
     }
   \]
   Applying Hermite reciprocity, this becomes
      \[
     \xymatrix{
       \displaystyle G_Z(u,v) \oo  \bw^{u+1}(\Sym^{a-1} U) \oo \bw^{v+1}(\Sym^{b-1} U) \ar[d] \\
     \displaystyle G_Z(u-1,v) \oo \bw^{u}(\Sym^{a-1} U ) \oo \bw^{v+1}(\Sym^{b-1} U) \oo \Sym^a U
     }
   \]
   To define this map, we use the comultiplication maps
   \begin{align*}
     G_Z(u,v) &\to G_Z(u-1,v) \otimes U\\
     \bw^{u+1}(\Sym^{a-1} U) &\to \bw^u(\Sym^{a-1} U) \oo \Sym^{a-1} U
   \end{align*}
   and then apply multiplication on the last two factors to get the factor $\Sym^a U$. The $v$-component is defined in a completely analogous way.

   \medskip

   Now consider $Z = \ms{D}_1$. The $u$-component takes the form
   \[
     \xymatrix{
       \D^{u+1} U \otimes \Sym^{a-u-1}(\D^{u+1} U) \otimes \Sym^{b-v}(\D^{v+1} U) \ar[d] \\
       \D^u U \otimes \Sym^{a-u}(\D^u U) \otimes \Sym^{b-v}(\D^{v+1} U) \otimes \Sym^a U
     }
   \]
   Applying Hermite reciprocity, this maps takes the form
   \[
     \xymatrix{
       \displaystyle \D^{u+1}U \oo  \bw^{u+1}(\Sym^{a-1} U) \oo \bw^{v+1}(\Sym^b U) \ar[d] \\
     \displaystyle \D^uU \oo \bw^{u}(\Sym^{a-1} U ) \oo \bw^{v+1}(\Sym^b U) \oo \Sym^a U
     }
   \]
   This is defined as before: we use the comultiplication maps
   \begin{align*}
     \D^{u+1} U &\to \D^u U \otimes U\\
     \bw^{u+1} (\Sym^{a-1} U) &\to \bw^u(\Sym^{a-1} U) \otimes \Sym^{a-1} U
   \end{align*}
   and multiply the last factors together. Under Hermite reciprocity, the $v$-component takes the form
   \[
          \xymatrix{
       \displaystyle \D^{u+1}U \oo  \bw^{u+1}(\Sym^{a-1} U) \oo \bw^{v+1}(\Sym^b U) \ar[d] \\
     \displaystyle \D^{u+1} U \oo \bw^{u+1}(\Sym^{a-1} U ) \oo \bw^{v}(\Sym^b U) \oo \Sym^b U
     }
   \]
   This is obtained by simply using the comultiplication map on the last exterior power factor. The definitions for $\ms{D}_2$ are completely analogous so we will omit the details.
   
   \medskip
   
   Now consider $Z=\ms{S}$. The $u$-component takes the form (the horizontal equalities are Hermite reciprocity)
   \[
     \xymatrix{
       \Sym^{a-u}(\D^{u+1} U) \oo \Sym^{b-v}(\D^{v+1} U) \ar@{=}[r] \ar[d] & \bw^{u+1}(\Sym^a U) \oo \bw^{v+1}(\Sym^b U) \ar[d]\\
       \Sym^{a-u+1}(\D^u U) \oo \Sym^{b-v}(\D^{v+1} U) \otimes \Sym^a U \ar@{=}[r] & \bw^u(\Sym^a U) \oo \bw^{v+1}(\Sym^b U) \oo \Sym^a U
       }
 \]
 The right vertical map is defined using exterior comultiplication.

  \subsection{Maps between the complexes}

 Applying $\Phi$ to \eqref{eqn:complexes}, we get a diagram
 \[
   \xymatrix{
     \Phi(\ms{C}) \ar[r]^0 & \Phi(\ms{S}) \\
   \Phi(\ms{B}) \ar[r] \ar[u]^{\Phi(\iota)} & \Phi(\ms{D}) \ar[u] \\
     \Phi(\ms{A}) \ar[r]^{\mathrm{id}} \ar[u]^{\Phi(\mu)} & \Phi(\ms{A}) \ar[u]
     }
   \]

   \begin{proposition} \label{prop:complex-diagram}
     All of the maps above are morphisms of double complexes.
   \end{proposition}

   \begin{proof}
     For the map $\Phi(\ms{A}) \to \Phi(\ms{B})$, compatibility in the $u$-direction amounts to the commutativity of the following diagram, which follows from coassociativity of the divided power comultiplication:
     \[
       \xymatrix{ \D^{u+v+2} U \ar[r] \ar[d] & \D^{u+1} U \oo \D^{v+1} U \ar[d] \\
         \D^{u+v+1} U \oo U \ar[r] & \D^u U \oo U \oo \D^{v+1} U } 
     \]
     Compatibility in the $v$-direction is analogous. The map $\Phi(\ms{A}) \to \Phi(\ms{D})$ is analogous.

Now consider $\Phi(\ms{B}) \to \Phi(\ms{C})$. We can factor it into two pieces: $\Phi(\ms{B}) \to \Phi(\ms{C}') \to \Phi(\ms{C})$. First, we have the formula
     \begin{align*}
       \iota_{u+1,v+1}^* \colon \D^{u+1} U \oo \D^{v+1} U &\to \D^{u} U \oo \D^{v} U\\
       x^{(i)} \oo x^{(j)} &\mapsto x^{(i-1)} \oo x^{(j)} - x^{(i)} \oo x^{(j-1)}.
     \end{align*}
     It follows from this explicit formula that the following diagram commutes, where the vertical maps are comultiplication:
     \[
       \xymatrix{
         \D^{u+1} U \oo \D^{v+1} U \ar[r]^{\iota_{u+1,v+1}^*} \ar[d] & \D^{u} U \oo \D^{v} U \ar[d] \\
         \D^u U \oo U \oo \D^{v+1} U \ar[r]^{\iota_{u,v}^*} & \D^{u-1}U \oo U \oo \D^{v} U
       }
     \]
     This implies compatibility of $\Phi(\ms{B}) \to \Phi(\ms{C}')$ in the $u$-direction, and the $v$-direction is similar. Also similarly, this can be used to prove compatibility of the map $\Phi(\ms{C}') \to \Phi(\ms{C})$. 

     Compatibility of $\Phi(\ms{D}) \to \Phi(\ms{S})$ reduces to Proposition~\ref{prop:comult-nu}.

     Finally, we prove compatibility of $\Phi(\ms{B}) \to \Phi(\ms{D})$. This map is a sum of two components, and the check is similar for both of them, so we will just explain the map $\Phi(\ms{B}) \to \Phi(\ms{D}_1)$. Compatibility in the $u$-direction is formal: the differential acts on different factors from the map $\Phi(\ms{B}) \to \Phi(\ms{D})$. Compatibility in the $v$-direction follows from Proposition~\ref{prop:comult-nu}.
   \end{proof}

\subsection{Homology of these complexes}

Consider the following data:
\[
 \begin{array}{|l||l|l|}
   \hline
   Z & E & F \\ \hline
   \ms{A} & (\Sym^{a-1} U)(-1) \oplus (\Sym^{b-1} U)(-1) & \mc{O}(a) \oplus \mc{O}(b) \\ \hline
   \ms{D}_1 & (\Sym^{a-1} U)(-1) \oplus \Sym^b U & \mc{O}(a) \\ \hline
   \ms{D}_2 & \Sym^a U \oplus (\Sym^{b-1}U)(-1) & \mc{O}(b) \\ \hline
 \end{array}
\]
In each case, we have a short exact sequence $0 \to E \to \Sym^a U \oplus \Sym^b U \to F \to 0$ of vector bundles over the projective variety $\bb{P}U$ such that
 \[
   \mathrm{tot}(\Phi(Z))_{i-1} = \bigoplus_{j \ge 0} \rH^j(\bb{P}U, \mc{O}(-2) \otimes \bw^{i+j} E) \otimes R.
 \]
Following \cite[\S 5]{weyman}, the terms on the right hand side have the structure of a minimal complex over $R$ by taking the derived pushforward of the Koszul complex on $E$. This describes the differentials that we have defined on the terms on the left hand side, so we conclude that the homology is
 \[
   \rH_{i-1}(\mathrm{tot}(\Phi(Z))) = \bigoplus_{d \ge 0} \rH^i(\bb{P}U, \mc{O}(-2) \otimes \Sym^d F).
 \]
Explicitly, we get
    \begin{align*}
     \rH_{-2}(\Phi(\ms{A})) &= \kk\\
     \rH_{-1}(\Phi(\ms{A})) &= \bigoplus_{d,e \ge 0} \Sym^{da+eb-2} U\\
     \rH_{-2}(\Phi(\ms{D}_1)) &= \kk\\
     \rH_{-1}(\Phi(\ms{D}_1)) &= \bigoplus_{d \ge 0} \Sym^{da-2} U\\
     \rH_{-2}(\Phi(\ms{D}_2)) &= \kk\\
     \rH_{-1}(\Phi(\ms{D}_2)) &= \bigoplus_{e \ge 0} \Sym^{eb-2} U.
   \end{align*}

   Next, $\Phi(\ms{N})$ and $\Phi(\ms{M})$ are respectively the homology of $\Phi(\ms{A}) \to \Phi(\ms{B}) \to \Phi(\ms{C})$ and $\Phi(\ms{A}) \to \Phi(\ms{D}) \to \Phi(\ms{S})$, and hence they inherit the structure of double complex. We now identify the corresponding total complexes.
   
   \begin{proposition}
     $\Phi(\ms{N})$ is the quotient complex of the minimal free resolution of the ideal $I$ of the rational normal scroll by the terms $\D^i U \oo \bw^{i+2}(\Sym^{a-1} U) \otimes R$ and $\D^i U \oo \bw^{i+2}(\Sym^{b-1} U) \otimes R$.
   \end{proposition}

   \begin{proof}
          First, we have a short exact sequence of double complexes
     \[
       0 \to \Phi(\ms{A}) \to \Phi(\ms{B}) \to \Phi(\ms{C}') \to 0
     \]
     as shown in the proof of Proposition~\ref{prop:complex-diagram}. Next, we have a short exact sequence
     \[
       0 \to \Phi(\ms{N}) \to \Phi(\ms{C}') \to \Phi(\ms{C}) \to 0
     \]
     where the last map is a morphism of double complexes, and hence $\Phi(\ms{N})$ inherits a double complex structure from being the kernel of this map.

     By coassociativity of comultiplication, the following diagram commutes
         \[
       \xymatrix{
         \D^{u+v} U \ar[r]^-{\mu_{u,v}^*} \ar[d] & \D^{u} U \oo \D^{v} U \ar[d] \\
         \D^{u+v-1} U \oo U \ar[r]^-{\mu_{u-1,v}^*} & \D^{u-1}U \oo U \oo \D^{v} U
       }
     \]
     where in the bottom map, the $U$ factor is not being used in $\mu_{u-1,v}^*$.
     
This implies that for the $u$-component of $\Phi(\ms{N})$, we use the comultiplication maps
   \begin{align*}
     \D^{u+v} U &\to \D^{u+v-1} U \otimes U\\
     \bw^{u+1}(\Sym^{a-1} U) &\to \bw^u(\Sym^{a-1} U) \otimes \Sym^{a-1} U
   \end{align*}
   together with the multiplication $U\otimes \Sym^{a-1} U \to \Sym^a U$. The $v$-component is defined similarly. This agrees with the quotient complex of the  minimal free resolution of the ideal $I$ of the rational normal scroll.
 \end{proof}

   \begin{proposition}
     $\Phi(\ms{M})$ is the first linear strand of the minimal free resolution of $\omega_B$.
   \end{proposition}
   
   \begin{proof}    
   First, the total complex of $\Phi(\ms{S})$ is a Koszul complex on $\Sym^a U \oplus \Sym^b U$ shifted by 2, so $\rH_{-2}(\Phi(\ms{S})) = \kk$ and all other homology vanishes.

   From the proof of Proposition~\ref{prop:omega-koszul}, we have a complex 
   \[
     0 \to \Phi(\ms{A}) \to \Phi(\ms{D}) \xrightarrow{f} \Phi(\ms{S}) \to R[-a,-b] \to 0
   \]
   whose middle homology is $\Phi(\ms{M})$. From the exact sequence
   \[
     0 \to \ker f \to \Phi(\ms{D}) \to \Phi(\ms{S}) \to R[-a,-b] \to 0
   \]
   and the calculations earlier, we conclude that
   \[
     \rH_{a+b-2}(\ker f) = R, \qquad \rH_{-2}(\ker f) = \kk, \qquad \rH_{-1}(\ker f) = \rH_{-1}(\Phi(\ms{D})).
   \]
   Next, from the short exact sequence $0 \to \Phi(\ms{A}) \to \ker f \to \Phi(\ms{M}) \to 0$, we get an exact sequence
   \[
     0 \to \rH_0(\Phi(\ms{M})) \to \rH_{-1}(\Phi(\ms{A})) \to \rH_{-1}(\ker f) \to \rH_{-1}(\Phi(\ms{M})) \to 0
   \]
   and $\rH_{a+b-2}(\Phi(\ms{M}))=R$. Since $\Phi(\ms{M})$ is concentrated in non-negative homological degrees, we conclude that
   \[
     \rH_0(\Phi(\ms{M})) = \bigoplus_{d,e \ge 1} \Sym^{da+eb-2} U.
   \]
   Next, $\Phi(\ms{M})_0 = \Sym^{a+b-2} U \otimes R$, so $\rH_0(\Phi(\ms{M}))$ is generated by its lowest degree term. We conclude that $\Phi(\ms{M})$ is the first linear strand of the minimal free resolution of $\omega_B$.
   \end{proof}
   
\begin{proof}[Proof of Theorem~\ref{thm:Ki2=Weyman}]
   Proposition~\ref{prop:complex-diagram} implies that we get a map of complexes
   \[
     F \colon \Phi(\ms{N}) \to \Phi(\ms{M}).
   \]
   On degree 0 components, this takes the form
   \[
     \Sym^{a-1} U \oo \Sym^{b-1} U \oo R \to \Sym^{a+b-2} U \oo R.
   \]
   This is the standard multiplication map, which follows from the explicit description of the map $\Phi(\ms{B}) \to \Phi(\ms{D})$. In particular, $F$ lifts the surjection $I \to \omega_B$, so that we can identify its maps with the induced maps on Tor.

To prove the last vanishing statement, we fix a bi-degree $(u+1,v+1)$, with $u+v=i$. We use Corollary~\ref{cor:vanishing-Weyman}, with $n_1=u+2$ and $n_2=v+2$. We have that $n_1+n_2-3 = u+v+1=i+1\leq\min(a,b)$, so the assumptions on the characteristic in the corollary are satisfied. We have moreover that $a-1-u\geq i-u=n_2-2$, and $b-1-v\geq i-v=n_1-2$, so $W^{(u+1,v+1)}_{a-1-u,b-1-v}=0$.
 \end{proof}

\section{Green's conjecture}\label{sec:green}

A canonical ribbon is a scheme which is a double structure on a rational normal curve. A hyperplane section of $\mc{X}(a,g-1-a)$ corresponds to a choice of polynomials $(f_1,f_2) \in \Sym^a U \oplus \Sym^{g-1-a} U$ and is a canonical ribbon if and only if $f_1,f_2$ is a regular sequence \cite[\S 2]{bayer-eisenbud}. These ribbons have Clifford index $a$ in the sense of \cite[\S 2]{bayer-eisenbud}. Since the homogeneous coordinate ring of $\mc{X}(a,g-1-a)$ is Cohen--Macaulay, it has the same graded Betti numbers as any canonical ribbon of Clifford index $a$ in $\PP^{g-1}$. Theorem~\ref{thm:Ki2=Weyman} then implies that the graded Betti numbers $\beta_{i,i+2}$ of canonical ribbons of Clifford index $a$ are 0 for $i<a$. 

\begin{proposition} \label{prop:smoothing}
Assume that the characteristic is not $2$. The canonical ribbons realized above can be smoothed out to a curve of gonality $a+2$ and Clifford index $a$. 
\end{proposition}

\begin{proof}
  In characteristic zero, this follows from the proof of \cite[Theorem 2]{fong}. The two key inputs for the proof are:
  \begin{itemize}
  \item \cite[Theorem 1]{fong}, which identifies ribbon structures with lines in the normal space to the hyperelliptic locus in the versal deformation space at some fixed hyperelliptic curve, and
  \item \cite[Theorem 2.1]{eisenbud-green}, which states that if a family of smooth curves of Clifford index $e$ degenerates to a ribbon, then the resulting Clifford index is $\le e$.
  \end{itemize}
  The proof of the first result goes through verbatim if we assume that 2 is invertible in $\kk$. To replace the latter result, it suffices to prove the following: if $C$ is the generic fiber of a flat family of smooth curves degenerating to one of the canonical ribbons above, then the Clifford index of $C$ is $\ge a$. To see this, we note first that the Hilbert series for a canonical ribbon is the same as the Hilbert series of a canonical curve of genus $a+b+1$, namely
  \[
    \frac{1+(a+b-1)t + (a+b-1)t^2 + t^3}{(1-t)^2},
  \]
which follows by passing to a hyperplane section in Proposition~\ref{prop:A-hilbert} (and using that $A$ is Cohen--Macaulay). Using \cite[Proposition~2.15]{bor-gre}, it follows that the Betti numbers in our family are upper semicontinuous. From the discussion above, we know that for the canonical ribbon $\beta_{i,i+2}=0$ for $i<a$, so we must also have that $\beta_{i,i+2}(C)=0$ for $i<a$. Using \cite[Corollary 9.7]{eisenbud-syzygies}, this implies that the Clifford index of $C$ is $\ge a$.
\end{proof}

We are now ready to deduce the generic Green's conjecture in each gonality, as follows.

\begin{theorem} \label{thm:gonality}
  Pick integers $a \ge 1$ and $g \ge 2a+1$. If the characteristic of $\kk$ is either $0$ or $p \ge a$, then there is a non-empty Zariski open subset of curves of gonality $a+2$ and Clifford index $a$ which satisfy Green's conjecture, i.e., $\beta_{i,i+2}=0$ for $i<a$ under the canonical embedding.
\end{theorem}

\begin{proof}
If $p=2$, then $a$ is 1 or 2. If $a \ge 1$, then $\beta_{0,2}=0$ for non-hyperelliptic curves by Noether's theorem. If $a=2$, then $\beta_{1,3}=0$ for non-trigonal curves by Petri's theorem. So for the remainder of the proof, we may assume that the characteristic is different from 2.
  
The condition $\beta_{i,i+2}=0$ for $i<a$ is open in the locus of curves of gonality $a+2$ in the moduli of curves of genus $g$. The condition that a curve of gonality $a+2$ has Clifford index $a$ is also open. Their intersection is non-empty by Proposition~\ref{prop:smoothing}.
\end{proof}

As a consequence, we solve \cite[Conjecture 0.1]{ES-K3carpets}:

\begin{corollary}
Let $\kk$ be a field of characteristic $p$ where either $p=0$ or $p \ge \lfloor \frac{g-1}{2} \rfloor$. Then a general curve of genus $g$ satisfies Green's conjecture, i.e., $\beta_{i,i+2}=0$ for $i<\lfloor (g-1)/2 \rfloor$.
\end{corollary}

\begin{proof}
  A curve of genus $g$ has Clifford index $\le \lfloor \frac{g-1}{2} \rfloor$ and for a general curve, this is the value of the Clifford index \cite[Theorem 8.16]{eisenbud-syzygies}. Hence the result follows from Theorem~\ref{thm:gonality}.
\end{proof}

\section*{Acknowledgments} 
The authors would like to thank Marian Aprodu, David Eisenbud, Gabi Farkas, Eric Riedl, Frank Schreyer, Claire Voisin, Jerzy Weyman, and Mengyuan Zhang for interesting discussions related to this project, for helpful suggestions, and for clarifications regarding the literature. We would also like to thank the anonymous referee for helping us improve the presentation of the paper. Experiments with the computer algebra software Macaulay2~\cite{M2} have provided many valuable insights.

\end{document}